\documentclass{amsart}

\usepackage[T1]{fontenc}
\usepackage{amsmath}
\usepackage{amsthm}
\usepackage{amsfonts}
\usepackage{amssymb}
\usepackage{enumerate}
\usepackage{mathrsfs}
\usepackage{multicol}
\usepackage{graphicx}
\usepackage{hyperref}
\usepackage[labelfont=rm]{caption}
\usepackage{subcaption}
\usepackage{tikz}
\usepackage{cleveref} 
\usepackage{enumitem}

\newif\iftikziii
\tikziiitrue

\iftikziii
\usetikzlibrary{matrix,calc,arrows.meta,knots,intersections,decorations.markings,cd}
\else
\usetikzlibrary{matrix,calc,knots,intersections,decorations.markings}
\fi

\newcommand\QQ{{\mathbb Q}}
\newcommand\CC{{\mathbb C}}
\newcommand\bn{{\mathbb N}}
\newcommand\PP{{\mathbb P}}
\newcommand\BB{{\mathbb B}}
\newcommand\bc{{\mathbb C}}
\newcommand\bp{{\mathbb P}}
\newcommand\bz{{\mathbb Z}}
\newcommand\cC{{\mathcal C}}
\newcommand\ZZ{{\mathbb Z}}
\newcommand\TT{{\mathbb T}}
\newcommand\cL{{\mathcal L}}
\newcommand\EE{{\mathbb E}}
\newcommand\bt{{\overline{\tau}}}

\DeclareMathOperator\PGL{\textrm{PGL}}
\DeclareMathOperator\im{Im}
\DeclareMathOperator\Hom{\textrm{Hom}}

\newtheorem{thm}{Theorem}[section]
\newtheorem{prop}{Proposition}[section]
\newtheorem{cor}{Corollary}[section]
\newtheorem{lema}{Lemma}[section]
\newtheorem{obs}{Remark}[section]
\theoremstyle{definition}

\newtheorem{question}{Question}[section]

\title{Triangular curves and cyclotomic Zariski tuples}
\author[E. Artal]{Enrique Artal Bartolo}

\author[J.I. Cogolludo]{Jos{\'e} Ignacio Cogolludo-Agust{\'i}n}
\address{Departamento de Matem\'aticas, IUMA\\ 
Universidad de Zaragoza\\ 
C.~Pedro Cerbuna 12\\ 
50009 Zaragoza, Spain} 
\email{artal@unizar.es,jicogo@unizar.es} 

\author[J.~Mart\'{i}n-Morales]{Jorge Mart\'{i}n-Morales}
\address{Centro Universitario de la Defensa-IUMA \\
Academia General Militar \\
Ctra.~de Huesca s/n. 50090, Zaragoza, Spain}
\email{jorge@unizar.es}
\urladdr{\url{http://cud.unizar.es/martin}}

\thanks{Partially supported by MTM2016-76868-C2-2-P and 
Gobierno de Arag{\'o}n (Grupo de referencia ``{\'A}lgebra y Geometr{\'i}a'')
cofunded by Feder 2014-2020 ``Construyendo Europa desde Arag\'on''.
The third author is also partially supported by FQM-333 ``Junta de Andaluc{\'\i}a''.}

\subjclass[2010]{14N20, 32S22, 14F35, 14H50, 14F45, 14G32}  

\keywords{Zariski pairs, arithmetic Zariski pairs, linking numbers}

\begin{document}
\begin{abstract}
The purpose of this paper is to exhibit infinite families of conjugate projective curves in a 
number field whose complement have the same abelian fundamental group, but are non-homeomorphic. 
In particular, for any $d\geq 4$ we find Zariski $\left(\left\lfloor\frac{d}{2}\right\rfloor+1\right)$-tuples 
parametrized by the $d$-roots of unity up to complex conjugation. As a consequence, for any divisor $m$ of $d$, 
$m\neq 1,2,3,4,6$, we find arithmetic Zariski $\frac{\phi(m)}{2}$-tuples with coefficients in the corresponding 
cyclotomic field. These curves have abelian fundamental group and they are distinguished using a linking invariant. 
\end{abstract}
\maketitle

\section*{Introduction}
Let us consider an algebraic variety $V$ admitting a set of equations whose coefficients belong in a number field,
$E\supset \QQ$. Consider $\sigma:E\to E$ a Galois transformation of $E$, then $\sigma$ acts on the defining equations
of $V$ to produce another algebraic variety $V^\sigma$ which is classically known as a \emph{conjugate} variety of $V$.
The same concept can be extended to any scheme. Note that the extension of $\sigma$ to $\CC$ is not necessarily a 
homeomorphism and hence $V$ and $V^\sigma$ are not necessarily homeomorphic. 
In~\cite{Serre-Examples}, Serre gave the first such example for a smooth surface whose conjugate has a non-isomorphic
fundamental group. Abelson~\cite{Abelson-conjugate} also found such examples for smooth projective varieties with isomorphic 
fundamental groups but different homotopy types. In addition, he also gave examples of conjugate, smooth 
homotopy equivalent, quasiprojective varieties that are not homeomorphic. Other examples by different authors can be 
found in~\cite{Artal-Carmona-ji-effective,Shimada-non-homeomorphic,Milne-Suh-Non-homeomorphic,Sh:18,Sh:19}.
In the previous examples either fundamental groups are non abelian or they had not been explicitly calculated.

In this paper we present examples of conjugate plane algebraic curves, or equivalently conjugate complements of curves,
whose fundamental groups are abelian (and hence isomorphic) yet they are not homeomorphic. 

In particular, we construct a family of irreducible plane projective curves whose deformation space is parametrized by 
the roots of unity of any given order $d$ up to complex conjugation. Associated with this equisingular space, another
family is constructed whose deformation space is also parametrized by the roots of unity. In this paper we focus on 
the topology both of their embedding and their complements. We prove that curves from different strata have both different 
embeddings and different topological types of their complements. This is in contrast to the simplicity of the fundamental
group of their complements, which is abelian. 

Particular examples of birationally equivalent curves, such as the union of smooth curves with three lines at maximal 
order flexes have already been considered in the literature~\cite{CC:08}, however not necessarily for their topological 
properties. The topology of their embedding was also recently studied by T.~Shirane in~\cite{Sh:18}, where a splitting 
graph is used (a generalization of an invariant introduced in~\cite{afg:17}). Shirane also extended this study to the 
union of smooth curves and any three non-concurrent lines in~\cite{Sh:19}. Some of our computations will apply to these 
curves as well showing that their complements are not homeomorphic and their fundamental groups are still abelian.
The invariant we are using is not original of this paper, see~\cite{ea:Toulouse,afg:17,benoit-meilhan,ben-shi}. 
Up to now, this invariant has been considered as an invariant of the pair given by the projective plane and the curve. 
In this work, we show that in good conditions it is an invariant of the complement.

An outline of the paper is as follows: in \S~\ref{sec:construction} we give a description of three families of curves 
whose equisingular space is non-connected and their connected components are parametrized by roots of unity of their
degree up to complex conjugation. Next, \S~\ref{sec:Zariski-tuples} is devoted to showing when curves in these families 
have different embeddings by means of a linking invariant (first defined in \cite{afg:17} by the first author and 
Florens-Guerville) that is sensitive to the embedding of a curve rather than just the fundamental group of its complement.
The purpose of \S~\ref{sec:pi1} is to explicitly describe different members of the different equisingular strata and
calculate the fundamental groups of their complements. Finally, in \S~\ref{sec:consecuencias} we extend our techniques
to the family studied by Shirane who showed that curves in different connected components of the equisingular strata 
have different embeddings. These families share a similar behavior, that is, their complements are also non-homeomorphic
yet their fundamental groups are abelian.

\subsection*{Acknowledgments}
The second and third authors want to thank the Fulbright Program (within the Jos\'e Castillejo and Salvador de Madariaga 
grants by Ministerio de Educaci\'on, Cultura y Deporte) for their financial support while writing this paper. They also want
to thank the University of Illinois at Chicago, especially Anatoly Libgober and Lawrence Ein for their warmth welcome and
support in hosting them.

\section{Construction of an equisingular family of curves}\label{sec:construction}

The purpose of this section is to construct three families of curves whose equisingular strata are not connected.
Such families have the property that conjugate curves belong to different connected components.

\subsection{The equisingular family \texorpdfstring{$\Sigma^{(d)}$}{Sigmad} of irreducible cuspidal curves of type
\texorpdfstring{$(d,d+1)$}{(d,d+1)}}
\mbox{}

For any given $d\in\bn$, $d>1$ let us denote by $\Sigma^{(d)}$ the realization space of (irreducible) 
plane projective curves in $\PP^2=\PP_\CC^2$ of degree~$2d$ having exactly three singular points whose local topological 
type is $u^d+v^{d+1}=0$. Such curves have genus 
$\binom{d-1}{2}$ and a classical example is the family $\Sigma^{(2)}$ 
of tricuspidal quartics, whose fundamental group of the complement is the braid group $\BB_3(\PP^1)$ on three strings on the 
sphere~$\mathbb S^2$ as was already known to Zariski~\cite{zr:29,Zariski-rational}.

Before we go any further note that $\Sigma^{(d)}\subset \PP^{N_d}$ for $N_d=\frac{d(d+3)}{2}$ is a smooth quasi-projective variety
and that the group $\PGL(3;\bc)$ of projective automorphisms of $\PP^2$ acts on $\Sigma^{(d)}$, maybe not freely. However, since
$\PGL(3;\bc)$ is a connected topological group, the number of connected components of $\Sigma^{(d)}$ and its quotient
$\PGL(3;\bc)\backslash\Sigma^{(d)}$ coincide. Moreover, since $\Sigma^{(d)}$ is smooth and quasi-projective, this number is also 
the number of irreducible components of these spaces. 

The following are properties of any curve $\cC\in\Sigma^{(d)}$:
\begin{enumerate}
 \item\label{prop:1}
 Since the singular points are locally irreducible, the curve~$\cC$ is \emph{globally} irreducible.
 \item\label{prop:2} 
 Let $P_x,P_y,P_z$ denote the singular points of $\cC$ and $L_z$ the line joining $P_x,P_y$. 
 Since the local intersection number of two curves at a point is at least the product of their multiplicities, 
 we have
\[
2d=\cC\cdot L_z\geq (\cC\cdot L_z)_{P_x}+(\cC\cdot L_z)_{P_y}
\geq d+d\Longrightarrow
\begin{cases}
\cC\cap L_z=\{P_x,P_y\}, \textrm{ and }\\
\cC\pitchfork_{P_x} L_z,\cC\pitchfork_{P_y} L_z.
\end{cases}
\]
Hence, $L_z$ is not tangent to $\cC$ at $P_x$ and $P_y$. Analogously occurs with $L_x$ (resp.~$L_y$) 
the line joining $P_y,P_z$ (resp.~$P_x,P_z$). Moreover, since they intersect pairwise at three different points, 
we deduce that $L_x$, $L_y$, and $L_z$ are three distinct and non-concurrent lines.
 \item\label{prop:3} 
 By the previous comment, there is a projective transformation~$\Phi$ such that $\Phi(L_x)=\{x=0\}$, $\Phi(L_y)=\{y=0\}$,
 and $\Phi(L_z)=\{z=0\}$ and thus $\Phi(P_x)=[1:0:0]$, $\Phi(P_y)=[0:1:0]$ and $\Phi(P_z)=[0:0:1]$. Hence, up to a projective
 change of coordinates, one can assume that $\cC$ satisfies these properties. Moreover, the projective transformation is not
 unique, that is, any diagonal projective automorphism can be further applied to $\cC$. 
\end{enumerate}

\begin{figure}[ht]
\begin{center}
\begin{tikzpicture}[scale=.5,vertice/.style={draw,circle,fill,minimum size=0.1cm,inner sep=0}]
\draw[fill, color=blue!10!white] (0,5)--(5,5)--(5,0)--cycle;
\draw[help lines,step=.5,dotted] (0,0) grid (10,10);
\draw[very thick] (0,10)--(0,0)--(10,0);
\foreach \x in {1,...,10}
 \node[vertice] at (\x/2,5) {};
\foreach \y in {1,...,10}
 \node[vertice] at (5,\y/2) {};
\foreach \z in {1,...,10}
 \node[vertice] at (5-\z/2,\z/2) {};
\node[anchor=north] at (3,-.5) {$\ldots$};
\node[anchor=north] at (5,-.1) {$(d,0)$};
\node[anchor=north] at (7.5,-.5) {$\ldots$};
\node[anchor=north] at (10,-.1) {$(2d,0)$};
\node[anchor=north east] at (0,0) {(0,0)};
\node[anchor=east] at (-.25,3) {$\vdots$};
\node[anchor=east] at (-.1,5) {$(0,d)$};
\node[anchor=east] at (-.25,7.5) {$\vdots$};
\node[anchor=east] at (-.1,10) {$(0,2d)$};
\draw[very thick,dashed] (10,0)--(0,10); 
\node[vertice] (a05) at (0,5) {};
\node[vertice] (a55) at (5,5) {};
\node[vertice] (a50) at (5,0) {};
\draw[very thick] (0,5)--(5,5)--(5,0)--cycle;
\draw[thick] (1,4.5)--(4.5,4.5)--(4.5,1)--cycle;
\end{tikzpicture}
\caption{Newton polygon}
\label{fig:np}
\end{center}
\end{figure}
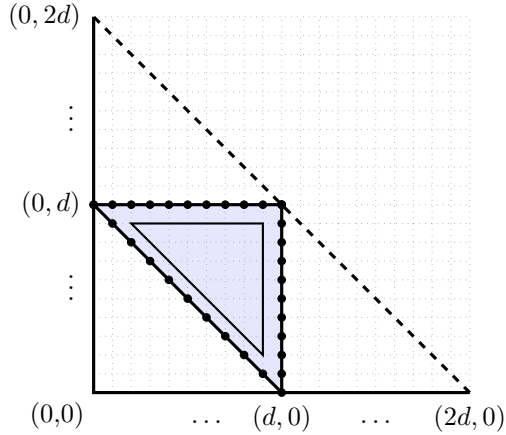

Let $F(x,y,z)=0$ be the equation of a curve~$\cC\in \Sigma^{(d)}$ satisfying these properties. Our purpose is 
to give the conditions the homogeneous polynomial $F$ of degree $2d$ should satisfy. A very useful tool to do 
this is to consider the Newton polygon of $f(x,y)=F(x,y,1)=\sum_{0\leq i+j\leq 2d} a_{ij}x^iy^j$, 
$$
N(f)=\textrm{ConvexHull}(\{(i,j)\in \ZZ^2_{\geq 0}\mid a_{ij}\neq 0\}).
$$
As a general restriction, since $F$ is a homogeneous polynomial of degree $2d$, $N(f)\subset T((0,0),(2d,0),(0,2d))$,
the triangle, or the convex hull of the three points. Moreover, since $\cC$ has multiplicity $d$ at $P_x$, $P_y$,
$P_z$ and the axes are not tangent, the Newton polygon is the triangle $T((0,d),(d,d),(d,0))$ of Figure~\ref{fig:np}. 
The transversality of the axes at the singular points implies that the vertices of this triangle correspond
to non-zero coefficients; we assume that the coefficient of $x^d y^d$ equals~$1$.

Moreover, since the Zariski tangent cone at the singularities has only one direction, each of the (quasi)-homogeneous
polynomials of $F$ associated with the sides of $N(f)$ must be a $d$-power of type $z^d(a_zx+b_zy)^d=(b_zz)^d(\lambda x+y)^d$,
$y^d(a_yz+b_yx)^d=(a_yy)^d(z+\mu x)^d$, and $x^d(a_xy+b_xz)^d=(b_xx)^d(\gamma y+z)^d$. According to property~\eqref{prop:3} above 
one can further assume that $b_z=\mu=\gamma=1$. Moreover, since the coefficient of $x^dy^d$ is 1, one has $a_y^d=b_x^d=1$ 
and hence the polynomial associated with the horizontal edge is $y^d(x+z)^d$, the one associated with the vertical edge is 
$x^d(y+z)^d$ and the one associated with the diagonal edge is $z^d(y+\zeta x)^d$, where $\zeta^d=1$ since $\lambda^d$ is the 
coefficient of $x^dz^d$ which is 1 in $x^d(y+z)^d$. 
Note that the isomorphism $[x:y:z]\mapsto[y:x:z]$ transforms $z^d(y+\zeta x)^d$ in $z^d(y+\bar{\zeta}x)^d$ after an appropriate 
diagonal isomorphism. Hence one obtains the following result.

\begin{prop}
\label{prop:sigma}
The realization set~$\Sigma^{(d)}$ decomposes in $\left(\left\lfloor\frac{d}{2}\right\rfloor+1\right)$ irreducible 
(and connected) components
$\Sigma^{(d)}_\zeta$ parametrized by 
$\TT_d^+:=\{\zeta\in\bc\mid\zeta^d=1,\im(\zeta)\geq 0\}$.
\end{prop}

\begin{proof}
As mentioned at the beginning of this section, we will prove the statement for the quotient space 
$\PGL(3;\bc)\backslash\Sigma^{(d)}$, that is, up to projective automorphisms.

The above discussion implies that if $\cC\in\Sigma^{(d)}$, then there exists a projective transformation $\Phi_1$ sending $\cC$
to a curve $\cC_1=\{F_\zeta(x,y,z)=0\}$ (in the same connected component of~$\Sigma^{(d)}$) such that the polynomials associated 
with the edges of its Newton polygon are $(x+z)^d, (y+z)^d, (y+\zeta x)^d$, $\zeta^d=1$. Moreover, after applying the symmetry with 
respect to $z=0$ we can assume that $\im(\zeta)\geq 0$.

Assume the existence of another $\Phi_2\in\PGL(3;\bc)$ such that $\Phi_2(\cC):=\cC_2=\{F_\omega(x,y,z)=0\}$, whose associated 
polynomials to the edges of its Newton polygon are $(x+z)^d, (y+z)^d, (y+\omega x)^d$, $\omega^d=1$, $\im(\omega)\geq 0$. 
We want to prove that $\omega=\zeta$. This implies the existence of a projective transformation sending 
$\{x+z=0, y+z=0, y+\zeta x=0\}$ to $\{x+z=0, y+z=0, y+\omega x=0\}$ and globally preserving the lines~$xyz=0$.
This transformation is a composition of a diagonal transformation and a permutation, and the result follows easily.

To end the proof one needs to show that for each $\zeta$ a curve $\cC\in \Sigma^{(d)}$ exists. Let us consider a polynomial 
$F$ whose Newton polygon is as in Figure~\ref{fig:np}, with the fixed polynomials in the edges of the triangle. This fact does 
not guarantee the topological type of the points $P_x,P_y,P_z$. Let us check for $P_z$. In order to have this topological type 
it is enough that the homogeneous part of degree~$d$ is coprime with the homogeneous part of degree~$d+1$. This is achieved 
for a generic choice of the coefficients of the inner triangle in Figure~\ref{fig:np} for $d>2$; the coefficients inside this 
inner triangle have no effect in the local type. For a generic choice no other singular points will arise 
(a confirmation of this fact will come later in this section). The connectivity for each $\zeta$ follows easily.

The case $d=2$ is special since the inner triangle mentioned above is empty. In that case, for $\zeta=-1$, one can check
that $x^2 y^2+ y^2 z^2 + x^2 z^2 +2 x y z(x+y-z)$ gives an appropriate equation; however, for $\zeta=1$, a non-reduced equation 
is obtained.
\end{proof}

\subsection{The equisingular family \texorpdfstring{$\tilde\Sigma^{(d)}$}{tildeSigmad} of triangular curves associated with \texorpdfstring{$\Sigma^{(d)}$}{Sigmad}}
\mbox{}

Let $\tilde{\Sigma}^{(d)}$ be the realization space of the reducible curves formed by an element of $\Sigma^{(d)}$
and the three lines joining the singular points, which we refer to as \emph{triangular}. 
This space coincides with the realization space of curves of degree~$2d+3$ 
with three singular points having the topological type of $u v((u+v)^d+v^{d+1})$. The argument starts by calculating the
multiplicity of intersection of a line joining two such singular points with a curve $\tilde \cC$ in $\tilde{\Sigma}^{(d)}$. 
Since each singular point has multiplicity $d+2$, B\'ezout's Theorem implies that the lines are irreducible components of  
$\tilde \cC$ and hence $\tilde \cC=L_1\cup L_2\cup L_3\cup \cC$ with $\cC\in \Sigma^{(d)}$. Hence combining the previous
discussion and Proposition~\ref{prop:sigma} one obtains the following.

\begin{cor}
The realization space~$\tilde{\Sigma}^{(d)}$ decomposes in $\left(\left\lfloor\frac{d}{2}\right\rfloor+1\right)$ irreducible 
(and connected) components $\tilde{\Sigma}^{(d)}_\zeta$, parametrized by $\zeta\in \TT_d^+$.
\end{cor}

\subsection{An equisingular family \texorpdfstring{$\hat{\Sigma}^{(d)}$}{hatSigmad} of triangular curves birationally equivalent to \texorpdfstring{$\tilde{\Sigma}^{(d)}$}{tildeSigmad}}
\mbox{}

Applying the canonical Cremona transformation we obtain a family $\hat{\Sigma}^{(d)}$, the realization space of curves formed 
by a smooth curve of degree~$d$ and three non-concurrent lines which are tangent at $d$-inflection points, which will also be 
referred to as \emph{triangular}. 

\begin{cor}
The realization space~$\hat{\Sigma}^{(d)}$ decomposes in $\left(\left\lfloor\frac{d}{2}\right\rfloor+1\right)$ irreducible 
(and connected) 
components $\hat{\Sigma}_\zeta^{(d)}$ parametrized by $\zeta\in \TT_d^+$.
\end{cor}

This fact was already studied in~\cite{CC:08,Sh:18}.

\section{Zariski tuples of conjugate curves}\label{sec:Zariski-tuples}

\subsection{A topological invariant of the embedding: the linking number}\label{sec:link}
\mbox{}

We review the results in~\cite{Sh:18} for completeness. Let $\hat \cC=\cC\cup L_1\cup L_2\cup L_3\in\hat{\Sigma}^{(d)}_\zeta$ 
for some $\zeta\in\TT_d^+$ where $\cC$ is the component of degree~$d$ of~$\hat \cC$. For the sake of simplicity the triangle 
$L_1\cup L_2\cup L_3$ is denoted by $\cL$. As mentioned above, after a suitable 
projective automorphism we can assume that $L_1=\{x=0\}$, $L_2=\{y=0\}$, $L_3=\{z=0\}$, and $\cC=\{F(x,y,z)=0\}$, where
\[
F(x,0,z)=(x+z)^d,\quad F(0,y,z)=(y+z)^d,\quad F(x,y,0)=(x+\zeta y)^d.
\]
Let us consider a model $S_\cC$ of the $d$-cyclic covering of $\bp^2$ ramified along~$\cC$, 
\[
S_\cC:=\{[x:y:z:t]\in\bp^3\mid t^d=F(x,y,z)\}.
\]
Any deck transformation of the covering should be identified with elements of $H_1(\bp^2\setminus\cC;\bz)=\ZZ/d\ZZ$,
or analogously, with $\Hom(H_1(\bp^2\setminus\cC;\bz);\CC^*)$. Consider a character $\xi:H_1(\bp^2\setminus\cC;\bz)\to\bc^*$ 
such that if $\mu$ is a meridian of~$\cC$, then $\xi(\mu)=\exp\left(\frac{2i\pi}{d}\right)=:\zeta_d$.
This character is associated with the monodromy automorphism $\sigma:S_\cC\to S_\cC$,
$\sigma([x:y:z:t]):=[x:y:z:\zeta_d t])$ as follows. If $\gamma$ is a cycle in $\bp^2\setminus\cC$ based at $[x_0:y_0:z_0]$ and 
$\tilde{\gamma}$ is the lift of $\gamma$ starting at $[x_0:y_0:z_0:t_0]\in S_\cC$, then the end point is 
$[x_0:y_0:z_0:\xi(\gamma) t_0]$.

Let us construct a cycle~$\gamma$ representing the homology class of a loop based at $[0:1:0]$ 
as a composition of paths $\gamma_1*\gamma_2*\gamma_3$ with $\gamma_i\subset L_i\setminus\cC$.
Note that $\gamma$ is supported on $\cL\setminus\cC$ and it is based at $[0:1:0]$ with the 
following trajectory:
\[
[0:1:0]\xrightarrow{\ \gamma_1\ }[0:0:1]\xrightarrow{\ \gamma_2\ }
[1:0:0]\xrightarrow{\ \gamma_3\ }[0:1:0].
\]
The fact that the contact order is~$d$ ensures that the choice of $\gamma$ does not affect the value of the character as 
will be checked below.

The lift $\tilde{\gamma}=\tilde\gamma_1*\tilde\gamma_2*\tilde\gamma_3$ of $\gamma$ starting at $[0:1:0:1]$ can be 
constructed as follows. First the path $\gamma_1$ is lifted, since $\cC\cap L_1$ has equations $x=t^d-(y+z)^d=0$, 
it has $d$ irreducible components, and since $\im(\gamma)\cap \cC=\emptyset$, the path $\gamma_1$ lifts to $\tilde\gamma_1$
in $S_1=\{x=t-(y+z)=0\}\subset S_\cC$. Note that the lift $\tilde{\gamma}_1$ ends at~$[0:0:1:1]$, which is the only 
point in $S_1$ and the fiber of $\gamma_1(1)=[0:0:1]$. Analogously, since $\cC\cap L_2$ is given by $\{y=t^d-(x+z)^d=0\}$, 
the lift $\tilde{\gamma}_2$ is supported on $\{y=t-(x+z)=0\}\subset S_\cC$ and it ends at~$[1:0:0:1]$. 
Finally, the equations of $\cC\cap L_3$ are $\{z=t^d-(x+\zeta y)^d=0\}$, and thus $\tilde{\gamma}_3$ is contained in 
$\{z=t-(x+\zeta y)=0\}\subset S_\cC$ whose end is~$[0:1:0:\zeta]$, i.e., $\xi(\gamma)=\zeta$, which agrees with the action 
of the monodromy~$\sigma$ on~$\gamma$ as discussed above.

Following~\cite{afg:17}, we will sketch how to obtain a topological invariant from these data.
Let $\hat\xi:H_1(\bp^2\setminus\hat \cC;\bz)\to\bc^*$ be the restriction of a character $\xi$ on $\bp^2\setminus\cC$ obtained 
as $\hat{\xi}:=\xi\circ i_*$, where $i:\bp^2\setminus\hat \cC\to\bp^2\setminus\cC$ is the inclusion.
Consider $\hat\gamma$ a cycle homologous to $\gamma$ in $H_1(\bp^2\setminus\cC;\bz)$ whose support is outside 
$\cL=L_1\cup L_2\cup L_3$. Then $\hat\xi(\gamma):=\hat\xi(\hat \gamma)=\zeta$ (since the character $\hat\xi$ is trivial 
on the meridians of the lines) and hence $\hat\xi(\gamma)$ is well defined since it does not depend on the representative 
$\hat \gamma$. Then, the value of $\hat\xi(\gamma)$ is an oriented topological invariant of $(\bp^2,\hat \cC^{\text{ord}},\hat\xi)$,
where the irreducible components in~$\hat \cC$ are ordered. Analogously, the set 
$\{\hat\xi(\gamma)^{\pm 1}\}$ is a topological invariant of 
$(\bp^2,\hat \cC,\hat\xi)$ when the conditions on orientation and ordering are removed.

\subsection{Main results}
\mbox{}

For completeness, we include the following result, already proven by T.~Shirane, to be refine 
in Theorem~\ref{thm:complement}.

\begin{thm}[\cite{Sh:18}]\label{thm:sh}
If $\hat\cC_{\zeta_i}\in\hat{\Sigma}^{(d)}_{\zeta_i}$, $\zeta_i\in \TT_d^+$, then the pairs 
$(\bp^2,\hat\cC_{\zeta_1})$ and $(\bp^2,\hat\cC_{\zeta_2})$ are homeomorphic if and only if $\zeta_1=\zeta_2$.
\end{thm}

\begin{proof}
Consider $\xi$ the character defined by $\xi(\mu)=\zeta_d$ for a meridian $\mu$ of $\cC$ and $\gamma$ the cycle described above.
Note that the pair $(\xi(\mu),\{\hat\xi(\gamma)^{\pm 1}\})$ becomes $(\zeta_d,\{\zeta^{\pm 1}_i\})$ when applied to the curve $\hat\cC_{\zeta_i}$. 
Since this is an invariant of $(\bp^2,\hat \cC_{\zeta_i},\hat\xi)$ and $\{\zeta_i,\bar\zeta_i\}\neq \{\zeta_j,\bar\zeta_j\}$, 
$\zeta_i,\zeta_j\in \TT_d^+$ if and only if $i=j$, this proves the \emph{only if} part.

The \emph{if} part follows from the fact that two curves in the same equisingular connected component have homeomorphic 
embeddings. The homeomorphism is obtained via the locally trivial fibration along a compact differentiable path joining both 
curves in the equisingular stratum of~$\hat\Sigma^{(d)}$.
\end{proof}

\begin{thm}\label{thm:complement}
Given $\hat\cC_{\zeta_i}\in\hat{\Sigma}^{(d)}_{\zeta_i}$, $\zeta_i\in \TT_d^+$, then
$\bp^2\setminus\hat\cC_{\zeta_1}$ and $\bp^2\setminus\hat\cC_{\zeta_2}$ are homeomorphic 
if and only if $\zeta_1=\zeta_2$.
\end{thm}

\begin{proof}
Let $U(\hat\cC_\zeta)$ be a compact regular neighborhood of $\hat\cC_\zeta$ and let $M_\zeta:=\partial U(\hat\cC_\zeta)$
denote its boundary manifold. This is a graph manifold associated with the plumbing graph of a minimal resolution 
of~$\hat\cC_\zeta$ (that is, the preimage of $\hat \cC_\zeta$ is normal crossing, but that can have self-intersections).

In principle, since $M_\zeta$ is associated with a normal bundle on the regular part of the normal crossing divisor
associated with the minimal resolution of $\hat \cC_\zeta$, the inclusion 
$M_\zeta\hookrightarrow \bp^2\setminus\hat\cC_{\zeta}$ induces only an isomorphism 
$\pi_1(M_\zeta)\to\pi_1^\infty(\bp^2\setminus\hat\cC_{\zeta})$. Nevertheless, a homeomorphism 
$\Psi:\bp^2\setminus\hat\cC_{\zeta_1}\to\bp^2\setminus\hat\cC_{\zeta_2}$ induces a homeomorphism of Milnor balls around 
the singular points of the curves and hence an isomorphism $\Psi^\infty:\pi_1(M_{\zeta_1})\to\pi_1(M_{\zeta_2})$.
From the properties of sufficiently large graph manifolds (see~\cite{wald:68}), this implies the existence of a homeomorphism
$\Psi_M:M_{\zeta_1}\to M_{\zeta_2}$ such that ${\Psi_M}_*=\Psi^\infty$.
Since the boundary manifolds $M_{\zeta_i}$ come with an extra structure due to the Milnor fibration around the singular points
this homeomorphism preserves the graph structure (see~\cite{wal:67b} and also the Appendix in~\cite{Neumann-Acalculus}). 
In particular, a meridian $\mu_{\zeta_1}$ of $\cC_{\zeta_1}$ must be sent to $\mu_{\zeta_2}^{\pm 1}$, $\pm 1$-power 
of a meridian of $\cC_{\zeta_2}$. Moreover the above cycle $\gamma_{\zeta_1}$ must be sent to a $\pm 1$-power of 
$\gamma_{\zeta_2}$. Hence $\zeta_1=\zeta_2$ if such a homeomorphism exists.
\end{proof}

\begin{cor}
If $\tilde\cC_{\zeta_i}\in\tilde{\Sigma}^{(d)}_{\zeta_i}$, $\zeta_i\in \TT_d^+$, for $i=1,2$, then 
$\bp^2\setminus\tilde\cC_{\zeta_1}$ and $\bp^2\setminus\tilde\cC_{\zeta_2}$ are homeomorphic if and only if~$\zeta_1=\zeta_2$.
\end{cor}

\begin{proof}
Recall that $\tilde\cC_{\zeta}\in\tilde{\Sigma}^{(d)}_{\zeta}$ is obtained from $\hat\cC_{\zeta}\in\hat{\Sigma}^{(d)}_{\zeta}$ 
via a standard Cremona transformation. This birational morphism produces a homeomorphism on their complements. 
Theorem~\ref{thm:complement} applies and the result follows.
\end{proof}

\section{Fundamental groups of complements of curves in \texorpdfstring{$\Sigma^{(d)}$}{Sigma d}, 
\texorpdfstring{$\tilde\Sigma^{(d)}$}{tilde Sigma d}, and \texorpdfstring{$\hat\Sigma^{(d)}$}{hat Sigma d}}
\label{sec:pi1}

Our goal in this section is to compute the fundamental group for any curve in $\tilde\cC\in\tilde{\Sigma}_{\zeta}^{(d)}$ 
and for any~$d$. We will do this using two different techniques.

\subsection{Lower degree cases}
\mbox{}

Fundamental groups of complements of curves in $\Sigma^{(d)}$ and $\tilde\Sigma^{(d)}$ or $\hat\Sigma^{(d)}$ are 
only well understood in the simplest cases, $d=2,3$. Here is a brief account of what is known.

For the case $d=2$, as mentioned above, it is classically known (see~\cite{zr:29}) that the fundamental group of the 
complement of the tricuspidal quartic $\cC\in\Sigma_{-1}^{(2)}$ is the finite metabelian group of order 12
$$\BB_3(\mathbb S^2)=
\langle \sigma_1,\sigma_2: \sigma_1\sigma_2\sigma_1=\sigma_2\sigma_1\sigma_2, \sigma_1\sigma_2^2\sigma_1=1\rangle
\cong \ZZ_3 \rtimes\ZZ_4.$$
It is an easy exercise to check that the fundamental group of the complement of a curve 
$\tilde\cC\in\tilde{\Sigma}_{-1}^{(2)}$ (a conic with three tangent lines) is the triangle Artin group~$(2,4,4)$. 

For the case $d=3$, the fundamental group of the complement of a sextic curve with $3\EE_6$ singularities depended 
on the connected component in $\Sigma^{(3)}$ it belongs (see~\cite{ac:98})
$$\pi_1(\PP^2\setminus \cC)=
\begin{cases}
\ZZ/2\ZZ*\ZZ/3\ZZ & \textrm{ if } \cC\in \Sigma_{1}^{(3)}\\
\ZZ/2\ZZ\times\ZZ/3\ZZ & \textrm{ if } \cC\in \Sigma_{\zeta_3}^{(3)},
\end{cases}
$$
where $\TT_3^+=\{1,\zeta_3\}$.
As for $\tilde{\Sigma}_{\zeta_3}^{(3)}$ it is also shown in~\cite{ac:98} that $\pi_1(\PP^2\setminus \tilde\cC)=\ZZ^3$ 
for $\tilde\cC\in\tilde{\Sigma}_{\zeta_3}^{(3)}$.

\subsection{Fermat curves}
\label{sec:fermat}
\mbox{}

In this section we present a way to calculate fundamental groups via finite coverings ramified along simpler curves. 
In this case we will use Fermat curves.

Let us consider the Kummer map $\Phi_d:\bp^2\to\bp^2$ given by $\Phi_d([x:y:z])=[x^d:y^d:z^d]$. 
Note that if $L=\{x+y+z=0\}$ then $\Phi_d^*(L)$ is a Fermat curve. The preimage of $L_x=\{y+z=0\}$
are $d$ tangent lines to $d$-inflection points, denoted by $L_{x,\tau}=\{y-\tau z=0\}$,
where $\tau^d=-1$. In the same way we consider $L_y=\{x+z=0\}$, $L_{y,\tau}=\{z-\tau x=0\}$
and also $L_z=\{x+y=0\}$, $L_{z,\tau}=\{x-\tau y=0\}$.

\begin{prop}
\label{prop:construction}
Let $\bt:=(\tau_1,\tau_2,\tau_3)$ where $\tau_i$ is a $d$-root of $-1$ and consider the following curve constructions:
\begin{enumerate}[label=\rm(\arabic{enumi})]
 \item\label{prop:uno}
$\hat\cC_{d,\bt,1}:=\Phi_d^*(L)\cup L_{x,\tau_1}\cup L_{y,\tau_2}\cup L_{z,\tau_3}$ with $\tau:=\tau_1\tau_2\tau_3$, $\zeta:=\tau^2$,
 \item\label{prop:dos} 
$\hat\cC_{d,\bt,2}:=\Phi_d^*(L)\cup L_{x,\tau_1}\cup L_{x,\tau_2}\cup L_{z,\tau_3}$ with $\zeta:=\tau_1\tau_2^{-1}$
and $\tau_1\neq \tau_2$,
 \item\label{prop:tres}
$\hat\cC_{d,\bt,3}:=\Phi_d^*(L)\cup L_{x,\tau_1}\cup L_{x,\tau_2}\cup L_{x,\tau_3}$ with $\zeta:=1$ and $\tau_i\neq \tau_j$, 
for $i\neq j$.
\end{enumerate}
Then $\hat\cC_{d,\bt,i}\in \hat{\Sigma}^{(d)}_{\zeta}$, $i=1,2$, where $\zeta\in \TT_{d}^+$.

For $i=3$, $\hat\cC_{d,\bt,3}\in \overline{\hat{\Sigma}^{(d)}_{1}}$ the closure of the component $\hat{\Sigma}^{(d)}_{1}$. 

Moreover, one can find a continuous equisingular family of curves $\{\hat\cC_t\}_{t\in(0,\varepsilon]}$ in $\hat{\Sigma}^{(d)}_{1}$ 
such that $\{\hat\cC_t\}_{t\in(0,\varepsilon]}\to\hat\cC_{d,\bt,3}$
where the triangle degenerates onto an ordinary triple point.
\end{prop}

\begin{proof}
For the family of curves in~\ref{prop:uno} consider the following change of coordinates:
\[
x_1=\tau x+y+z,\quad
y_1=\tau_3^{-1}(\tau x+y+\tau z),\quad
z_1=\tau_2(\tau x+z+\tau^{-1} y).
\]
Let us consider the homogeneous polynomial 
$F(x,y,z)=x_1^d+y_1^d+z_1^d$. Then
\[
F(0,y,z)=(y+z)^d,\quad
F(x,0,z)=(z+x)^d,\quad
F(x,y,0)=(\zeta x+y)^d,
\]
and the statement follows.

The proof for the family of curves in construction~\ref{prop:dos} is analogous considering the following change of coordinates:
\[
x_1= x+ y+z,\quad
y_1=\tau_3(x+y),\quad
z_1=\tau_3(\tau_1 x+\tau_2 y)
\]
and  
$F(x,y,z)=x_1^d+y_1^d+z_1^d$. Then
\[
F(0,y,z)=(y+z)^d,\quad
F(x,0,z)=(x+z)^d,\quad
F(x,y,0)=(x+\zeta^{-1} y)^d.
\]

For construction~\ref{prop:tres} consider the following family
\[
\{(x-\tau_1 y)(x-\tau_2 y)(x-\tau_3 y-t z)\left(x^d+y^d+z^d-t z\frac{x^d+y^d}{x-\tau_3 y}\right)=0\}_{t\in(0,\varepsilon]}.
\]
One can check that for $0<t\leq\varepsilon$ these curves are union of a smooth curve of degree~$d$ and three tangent lines
at aligned $d$-flexes. The fact that these flexes are aligned only occurs if the curve is in $\hat{\Sigma}^{(d)}_{1}$.
\end{proof}

\begin{obs}
Note that construction~\ref{prop:uno} in Proposition~\ref{prop:construction} produces representatives in 
$\hat{\Sigma}^{(d)}_{\zeta}$ for all $\zeta\in\TT_d^+$ in the odd case, whereas construction~\ref{prop:dos}
produces representatives in $\hat{\Sigma}^{(d)}_{\zeta}$ as long as $\zeta\neq 1$. The only representative not
produced directly this way is one in the component $\hat{\Sigma}^{(d)}_{1}$ for $d$ even.
\end{obs}

\subsection{Fundamental group computation}
\mbox{}

We start with the fundamental group of a simple line arrangement, whose computation
can be done using its real affine picture shown in Figure~\ref{fig:lines1}
where $\ell_x=\{y=0\}$, $\tilde\ell_x=\{y+z=0\}$, $\ell_y=\{x=0\}$, $\tilde\ell_y=\{x+z=0\}$, $\ell=\{x+y+z=0\}$,
and $\ell_\infty=\{z=0\}$. 

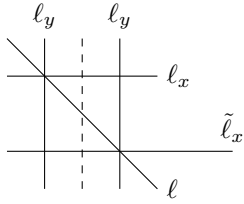
\begin{figure}[ht]
\begin{center}

\begin{tikzpicture}[vertice/.style={draw,circle,fill,minimum size=0.2cm,inner sep=0}]
\tikzset{%
  suma/.style args={#1 and #2}{to path={%
 ($(\tikztostart)!-#1!(\tikztotarget)$)--($(\tikztotarget)!-#2!(\tikztostart)$)%
  \tikztonodes}}
} 
\coordinate (XY) at (0,0);
\coordinate (X) at (-1,0);
\coordinate (Y) at (0,-1);
\coordinate (Z) at (-1,-1);
\coordinate (X1) at (-1,1);
\draw[suma=.5 and .5] (X) to (XY) node[right] {$\ell_x$};
\draw[suma=.5 and .5] (Y) to (XY) node[above] {$\ell_y$} ;
\draw[suma=.5 and .5,dashed] ($.5*(XY)+.5*(X)$) to ($.5*(Y)+.5*(Z)$);
\draw[suma=.5 and .5] (X) to (Y) node[right] {$\ell$};
\draw[suma=.5 and .5] (Z) to (X)  node[above] {$\tilde\ell_{y}$};
\draw[suma=.5 and 1.5] (Z) to (Y) node[above] {$\tilde\ell_{x}$};
\end{tikzpicture}
\caption{$6$-line arrangement}
\label{fig:lines1}
\end{center}
\end{figure}

\begin{prop}
The fundamental group $G$ of the complement of the line arrangement 
$\ell_x\cup \ell_y\cup \tilde\ell_x\cup \tilde\ell_y\cup \ell\cup \ell_\infty\subset \bp^2$ has generators 
$\gamma_x,\gamma_y,\gamma_\ell,\tilde\gamma_{x},\tilde\gamma_{y}$ and relations
\begin{multicols}{3}
\begin{enumerate}[label=\rm(G\arabic*)]
\item\label{rel1-xy} $[\tilde\gamma_{x},\tilde\gamma_{y}]=1$,
\item\label{rel1-tilde-x1} $[\gamma_y\cdot\gamma_\ell,\tilde\gamma_{x}]=1$,
\item\label{rel1-tilde-x2} $[\tilde\gamma_{x}\cdot \gamma_y,\gamma_\ell]=1$,
\end{enumerate}
\end{multicols}
\vspace{-5mm}
\begin{multicols}{3}
\begin{enumerate}[label=\rm(G\arabic*)]
\setcounter{enumi}{3}
\item\label{rel1-tilde-l1} $[\gamma_y,\gamma_x]=1$,
\item\label{rel1-tilde-y1} $[\gamma_x\cdot\tilde\gamma_{y},\gamma_\ell]=1$,
\item\label{rel1-tilde-y2} $[\gamma_\ell\cdot \gamma_x,\tilde\gamma_{y}]=1$.
\end{enumerate}
\end{multicols}
Moreover, the orbifold fundamental group~$\tilde G$ of the complement of $\ell\cup \tilde\ell_x\cup \tilde\ell_y\subset \PP^2$ 
with orbifold locus of order~$d$ along $\ell_x\cup \ell_y\cup \ell_\infty$, is obtained as a quotient of $G$ by the 
normal subgroup generated by the relations
\begin{multicols}{3}
\begin{enumerate}[label=\rm(G\arabic*)]
\setcounter{enumi}{6}
\item\label{rel1-x} $\gamma_x^d=1$,
\item\label{rel1-y} $\gamma_y^d=1$,
\item\label{rel1-z} 
$\gamma_\infty^d=1$,
\end{enumerate}
\end{multicols}
where $\gamma_\infty=(\tilde\gamma_{x}\cdot\gamma_\ell\cdot \gamma_x\cdot \tilde\gamma_{y}\cdot \gamma_y)^{-1}$.
\end{prop}

It is important to stress that the previous presentations have a geometrical meaning which will be relevant in the 
upcoming calculations. The generators $\gamma_x$, $\gamma_\ell$, and $\tilde\gamma_x$ on the dashed vertical line in 
Figure~\ref{fig:lines1} are meridians around the lines $\ell_x$, $\ell$, and $\tilde\ell_x$ respectively. They are
chosen as a geometrical basis on the punctured line and their product $\tilde\gamma_{x}\cdot\gamma_\ell\cdot \gamma_x$
is the inverse of a meridian around the point at infinity of such a line. The remaining generators $\tilde\gamma_{y}$, 
$\gamma_y$ are meridians around the lines $\tilde\ell_y$ and $\ell_y$ respectively, on a horizontal line in 
Figure~\ref{fig:lines1}.

namely, $\gamma_x$ (resp.~$\gamma_y$) is a meridian of $\ell_x$ (resp.~$\ell_y$), 
$\tilde\gamma_x$ (resp.~$\tilde\gamma_y$) is a meridian of $\tilde\ell_x$ (resp.~$\tilde\ell_y$),
and so are $\gamma_\ell$ with respect to $\ell$ and $\gamma_\infty$ with respect to $\ell_\infty$.

Following the construction given in section~\ref{sec:construction} one needs to study the Kummer cover 
\[
\sigma:\PP^2\to\PP^2,\qquad [x:y:z]\mapsto[x^d:y^d:z^d].
\]
Note that $\sigma$ is an abelian cover of order $d^2$ ramified along 
$\ell_x\cup \ell_y\cup \ell_\infty$ with ramification index $d$ whose group of deck transformations is 
$\ZZ/d\ZZ\times \ZZ/d\ZZ$. A possible strategy is to decompose $\sigma$ in two cyclic covers. 
The first one $\sigma_1:S\to\bp^2$ can be seen as an orbifold covering ramified along the lines $\ell_y$ and $\ell_\infty$. 
The second one $\sigma_2:\bp^2\to S$ can be seen as an orbifold covering ramified along the preimages by~$\sigma_1$ 
of $\ell_x$ and $[1:0:0]$. If we remove the line arrangement, the covering $\sigma_1$ is defined by an epimorphism 
$\rho_1:\tilde G\to\mu_d$, $\rho(\gamma_y):=\zeta$, $d$-root of unity, and
$\rho(\gamma_x)=\rho(\tilde\gamma_{x})=\rho(\tilde\gamma_{y})=\rho(\gamma_\ell)=1$. 
Let $\hat{K}:=\ker\rho_1$. Note that $\sigma_1^{-1}(\tilde\ell_y)$ breaks in $d$ irreducible 
components, whose meridians are $\gamma_y^j\cdot\tilde\gamma_{y}\cdot \gamma_y^{j-1}$,
$0\leq j<d$. By Proposition~\ref{prop:construction}\ref{prop:dos}, we are interested in
keeping only one of these lines, say the one associated with~$\tilde\gamma_{y}$. Let $\hat{K}_1$
be the quotient of $\hat{K}$ obtained by killing the remaining meridians. It is the 
orbifold fundamental group of the complement in~$S$ of the line arrangement with orbifold
structure at the preimage of $\ell_x$ and $[1:0:0]$.

\begin{lema}
The group $\hat{K}_1$ is generated by $\gamma_x,\gamma_\ell,\tilde\gamma_{x},\tilde\gamma_{y}$
with relations:
\begin{multicols}{3}
\begin{enumerate}[label=\rm(K\arabic*)]
\item\label{rel1-x-k} $\gamma_x^d=1$,
\item\label{rel1-z-k} 
$(\tilde\gamma_{x}\cdot\gamma_\ell\cdot \gamma_x)^d\cdot\tilde\gamma_{y}=1$,
\item\label{rel1-xy-k} $[\tilde\gamma_{x},\tilde\gamma_{y}]=1$,
\end{enumerate}
\end{multicols}
\vspace{-7mm}
\begin{multicols}{3}
\begin{enumerate}[label=\rm(K\arabic*)]
\setcounter{enumi}{3}
\item\label{rel1-tilde-x1-k} $\left[(\ell\cdot\tilde\gamma_{x})^d,\tilde\gamma_{x}\right]=1$,
\item\label{rel1-tilde-y1-k} $[\gamma_x\cdot\tilde\gamma_{y},\gamma_\ell]=1$,
\item\label{rel1-tilde-y2-k} $[\gamma_\ell\cdot \gamma_x,\tilde\gamma_{y}]=1$,
\end{enumerate}
\end{multicols}
\vspace{-3mm}
\begin{enumerate}[label=\rm(K\arabic*)]
\setcounter{enumi}{6}
 \item \label{rel1-j-k} $\left[\gamma_x,(\ell\cdot\tilde\gamma_{x})^{-j}\cdot\gamma_\ell\cdot(\gamma_\ell\cdot\tilde\gamma_{x})^j\right]=1$,  for  $0<j<d$.
\end{enumerate}

\end{lema}

\begin{proof}
We use Reidemeister-Schreier method. For each generator $\alpha$ of $\tilde G$, different from $\gamma_y$,
the generators are $\alpha_j:=\gamma_y^j\cdot\alpha\cdot \gamma_y^{-j}$, $j=0,1,\dots,d-1$. The element
$\gamma_y^d$ should be also a generator, which can be avoided by~\ref{rel1-y}.
Using~\ref{rel1-tilde-l1}, we obtain
\begin{equation*}
\gamma_{x_j}:=\gamma_y^j\cdot \gamma_x\cdot \gamma_y^{-j}=\gamma_x,
\end{equation*}
and we keep only $\gamma_{x_0}=\gamma_x$ as generator from $\gamma_{x_0},\dots,\gamma_{x_{d-1}}$.
The relation~\ref{rel1-x} is kept (all its conjugates provide the same relation).
By the definition of $\hat{K}_1$, we keep only $\tilde\gamma_{y}=\tilde\gamma_{{y}_0}$ as
$\tilde\gamma_{{y}_j}=1$, for $0<j<d$.

The relation~\ref{rel1-xy} also remains and it is unique. We use relations~\ref{rel1-tilde-x1}
and~\ref{rel1-tilde-x2} to eliminate some generators since
\begin{equation*}
\gamma_{\ell_j}=\gamma_y^j\cdot\gamma_\ell\cdot \gamma_y^{-j}=
(\gamma_\ell\cdot\tilde\gamma_{x})^{-j}\cdot\gamma_\ell\cdot(\gamma_\ell\cdot\tilde\gamma_{x})^j,
\quad 
\tilde\gamma_{x_j}=\gamma_y^j\cdot\tilde\gamma_{x}\cdot \gamma_y^{-j}=
(\gamma_\ell\cdot\tilde\gamma_{x})^{-j}\cdot\tilde\gamma_{x}\cdot(\gamma_\ell\cdot\tilde\gamma_{x})^j,
\end{equation*}
producing the relation
\begin{equation*}
\left[(\gamma_\ell\cdot\tilde\gamma_{x})^d,\tilde\gamma_{x}\right]=1.
\end{equation*}
The relations given by~\ref{rel1-tilde-y1}
and~\ref{rel1-tilde-y2} are kept and produce for $0<j<d$:
\begin{equation*}
\left[\gamma_x,(\gamma_\ell\cdot\tilde\gamma_{x})^{-j}\cdot\gamma_\ell\cdot(\gamma_\ell\cdot\tilde\gamma_{x})^j\right]=1.
\end{equation*}
Finally the relation~\ref{rel1-z} becomes
\begin{equation*}
(\tilde\gamma_{x}\cdot\gamma_\ell\cdot \gamma_x)^d\cdot\tilde\gamma_{y}=1.
\qedhere
\end{equation*}
\end{proof}

The covering $\rho_2$ is defined by the morphism $\rho_2:\hat{K}_1\to\mu_d$ such that 
$\rho_2(\gamma_x)=\zeta$ and the other generators are sent to~$1$. 

\begin{lema}
The group $\tilde{K}:=\ker\rho_2$ is generated by
$\gamma_\ell,\tilde\gamma_{y},\tilde\gamma_{x_i}:=\gamma_x^i\cdot\tilde\gamma_{x}\cdot \gamma_x^{-i}$, $0\leq i<d$
with relations:
\begin{enumerate}[label=\rm($\tilde{\text{K}}$\arabic*)]
\item $\left[(\tilde\gamma_{y}\cdot\gamma_\ell)^d,\tilde\gamma_{y}\right]=1$,
\item\label{tk2} $[\tilde\gamma_{x_i},\left(\tilde\gamma_{y}\cdot\gamma_\ell\right)^{-i}\cdot\tilde\gamma_{y}\cdot
\left(\tilde\gamma_{y}\cdot\gamma_\ell\right)^{i}]=1$ for $0\leq i<d$,
\item\label{tk3} $\prod_{j=0}^{d-1}(\tilde\gamma_{x_j}\cdot(\tilde\gamma_{y}\cdot\gamma_\ell)^{-j}\cdot\gamma_\ell
\cdot (\tilde\gamma_{y}\cdot\gamma_\ell)^{j})\cdot \tilde\gamma_{y}=1$,
\item\label{tk4} $\left[(\left(\tilde\gamma_{y}\cdot\gamma_\ell\right)^{-i}\cdot\gamma_\ell\cdot
\left(\tilde{y}\cdot\gamma_\ell\right)^{i}\cdot\tilde\gamma_{x_i})^d,\tilde\gamma_{x_i}\right]=1$
for $0\leq i<d$,
\item\label{tk5} $1=[\gamma_\ell,(\gamma_\ell\cdot\left(\tilde\gamma_{y}\cdot\gamma_\ell\right)^{i}
\cdot\tilde\gamma_{x_i}\cdot\left(\tilde\gamma_{y}\cdot\gamma_\ell\right)^{-i})^j\cdot \left(\tilde\gamma_{y}\cdot\gamma_\ell\right)^{-1}\cdot (\gamma_\ell\cdot\left(\tilde\gamma_{y}\cdot\gamma_\ell\right)^{i+1}\cdot\tilde\gamma_{x_{i+1}}\cdot
\left(\tilde\gamma_{y}\cdot\gamma_\ell\right)^{-i-1})^{-j}]$, for $0<j<d$ and $0\leq i<d$.
\end{enumerate}
\end{lema}

\begin{proof}
We apply again Reidemeister-Schreier algorithm. We should start 
with generators $\gamma_{\ell_j}=\gamma_x^j\cdot\gamma_\ell \gamma_x^{-j}$,  $\tilde\gamma_{y_j}=
\gamma_x^j\cdot\tilde\gamma_{y} \gamma_x^{-j}$,
$\tilde\gamma_{x_j}=\gamma_x^j\cdot\tilde\gamma_{x} \gamma_x^{-j}$, and $\gamma_x^d$. The generator  
$\gamma_x^d$ drops needed because of~\ref{rel1-x-k}.
Using~\ref{rel1-tilde-y1-k} and~\ref{rel1-tilde-y2-k}, we eliminate some generators
\begin{equation*}
\gamma_{\ell_i}=\gamma_x^i\cdot\gamma_\ell\cdot \gamma_x^{-i}=
\left(\tilde\gamma_{y}\cdot\gamma_\ell\right)^{-i}\cdot\gamma_\ell\cdot \left(\tilde\gamma_{y}\cdot\gamma_\ell\right)^{i},\quad
\tilde\gamma_{y_i}=\gamma_x^i\cdot\tilde\gamma_{y}\cdot \gamma_x^{-i}=
\left(\tilde\gamma_{y}\cdot\gamma_\ell\right)^{-i}\cdot\tilde\gamma_{y}\cdot
\left(\tilde\gamma_{y}\cdot\gamma_\ell\right)^{i},
\end{equation*}
and obtain the relation
\begin{equation*}
\left[(\tilde\gamma_{y}\cdot\gamma_\ell)^d,\tilde\gamma_{y}\right]=1.
\end{equation*}
The generator statement has been proven.
The relation~\ref{rel1-xy-k} becomes $d$ relations ($0\leq i<d$)
\[
[\tilde\gamma_{x_i},\left(\tilde\gamma_{y}\cdot\gamma_\ell\right)^{-i}\cdot\tilde\gamma_{y}\cdot
\left(\tilde\gamma_{y}\cdot\gamma_\ell\right)^{i}]=1.
\]
Using the above commutation relations, the relation~\ref{rel1-z-k} becomes
\[
\prod_{j=0}^{d-1}(\tilde\gamma_{x_j}\cdot(\tilde\gamma_{y}\cdot\gamma_\ell)^{-j}\cdot\gamma_\ell\cdot 
(\tilde\gamma_{y}\cdot\gamma_\ell)^{j})\cdot \tilde\gamma_{y}=1.
\]
The relation~\ref{rel1-tilde-x1-k} becomes $d$ relations ($0\leq i<d$)
\begin{equation*}
\left[(\left(\tilde\gamma_{y}\cdot\gamma_\ell\right)^{-i}\cdot\gamma_\ell\cdot
\left(\tilde\gamma_{y}\cdot\gamma_\ell\right)^{i}\cdot\tilde\gamma_{x_i})^d,\tilde\gamma_{x_i}\right]=1.
\end{equation*}
Finally the relations \ref{rel1-j-k} produce $d(d-1)$ relations for $0<j<d$ and $0\leq i<d$:
\begin{gather*} 
\!(\left(\tilde\gamma_{y}\cdot\gamma_\ell\right)^{-i}\cdot\gamma_\ell\cdot
\left(\tilde\gamma_{y}\cdot\gamma_\ell\right)^{i}\cdot\tilde\gamma_{x_i})^{-j}\cdot
\left(\tilde\gamma_{y}\cdot\gamma_\ell\right)^{-i}\cdot\gamma_\ell\cdot
\left(\tilde\gamma_{y}\cdot\gamma_\ell\right)^{i}\cdot (\left(\tilde\gamma_{y}\cdot
\gamma_\ell\right)^{-i}\cdot\gamma_\ell\cdot\left(\tilde\gamma_{y}\cdot\gamma_\ell\right)^{i}\cdot \tilde\gamma_{x_i})^j=\\
(\left(\tilde\gamma_{y}\!\cdot\!\gamma_\ell\right)^{-i-1}\!\!\!\cdot\!\gamma_\ell\!\cdot\!\left(\tilde\gamma_{y}\cdot
\gamma_\ell\right)^{i+1}\!\!\cdot\tilde\gamma_{x_{i+1}})^{-j}\!\cdot\!\left(\tilde\gamma_{y}\!\!\cdot
\gamma_\ell\right)^{-i-1}\!\!\!\cdot\!\gamma_\ell\!\cdot\!\left(\tilde\gamma_{y}\!\cdot\!\gamma_\ell\right)^{i+1}\!\!\cdot\!
(\!\left(\tilde\gamma_{y}\!\cdot\!\gamma_\ell\right)^{-i-1}\!\!\!\cdot\!\gamma_\ell\!\cdot\!\left(\tilde\gamma_{y}\!\cdot\!
\gamma_\ell\right)^{i+1}\!\!\cdot\!\tilde\gamma_{x_{i+1}})^j,
\end{gather*}
that is,
\begin{gather*}
1=[\gamma_\ell,(\gamma_\ell\cdot\left(\tilde\gamma_{y}\cdot\gamma_\ell\right)^{i}\cdot\tilde\gamma_{x_i}\cdot
\left(\tilde\gamma_{y}\cdot\gamma_\ell\right)^{-i})^j\cdot \left(\tilde\gamma_{y}\cdot\gamma_\ell\right)^{-1}\cdot
(\gamma_\ell\cdot\left(\tilde\gamma_{y}\cdot\gamma_\ell\right)^{i+1}\cdot\tilde\gamma_{x_{i+1}}\cdot
\left(\tilde\gamma_{y}\cdot\gamma_\ell\right)^{-i-1})^{-j}].
\qedhere
\end{gather*}
\end{proof}

\subsection{Main results on fundamental groups}
\mbox{}

In this section we summarize the main consequences in terms of fundamental groups for the curves in the 
equisingular stratum $\hat{\Sigma}^{(d)}_\zeta$.

\begin{thm}\label{thm:sigmazeta}
Let $\cC\in\hat{\Sigma}^{(d)}_\zeta$, $\zeta\neq 1$, $d>2$. Its fundamental group is abelian, in particular 
$\pi_1(\PP^2\setminus \cC)=\ZZ^3$.
\end{thm}

\begin{proof}
It is enough to prove the statement for a curve  as in
Proposition{\rm~\ref{prop:construction}\ref{prop:dos}}, for which its fundamental group is  $K_h$, $0<h<d$, the quotient of $\hat{K}_1$ obtained by
\emph{killing} $\tilde\gamma_{x_j}$, if $j\neq 0,h$. We can assume $h\geq 2$.

The relation~\ref{tk2} produces $[\tilde\gamma_{x_0},\tilde\gamma_{y}]=[\tilde\gamma_{x_h},(\tilde\gamma_{y}\cdot
\gamma_\ell)^{-h}\cdot\tilde\gamma_{y}\cdot(\tilde\gamma_{y}\cdot\gamma_\ell)^{h}]=1$. The relation~\ref{tk3} becomes
\[
\tilde\gamma_{x_0}\cdot(\tilde\gamma_{y}\cdot\gamma_\ell)^{h}\cdot
\tilde\gamma_{x_h}\cdot(\tilde\gamma_{y}\cdot\gamma_\ell)^{-h}\cdot(\tilde\gamma_{y}\cdot\gamma_\ell)^{d}\cdot
\tilde\gamma_{y}^{1-d}=1
\]
and $\tilde\gamma_{x_h}$ can be taken out from the generators. The second relation above is equivalent to the commutator
$[(\tilde\gamma_{y}\cdot\gamma_\ell)^{d},\tilde\gamma_{x_0}]=1$.
For \ref{tk4} we obtain
\[
\left[(\gamma_\ell
\cdot\tilde\gamma_{x_0})^d,\tilde\gamma_{x_0}\right]=
\left[(\left(\tilde\gamma_{y}\cdot\gamma_\ell\right)^{-h}\cdot\gamma_\ell\cdot
\left(\tilde\gamma_{y}\cdot\gamma_\ell\right)^{h}\cdot\tilde\gamma_{x_h})^d,\tilde\gamma_{x_h}\right]=1.
\]
Let us consider \ref{tk5} for $j=1,2$:
\[
[\gamma_\ell,\left(\tilde\gamma_{y}\cdot\gamma_\ell\right)^{i}\cdot\tilde\gamma_{x_i}\cdot
\tilde\gamma_{x_{i+1}}^{-1}\cdot\left(\tilde\gamma_{y}\cdot\gamma_\ell\right)^{-i-1}]=1
\]
and
\[
1=[\gamma_\ell,
\left(\tilde\gamma_{y}\cdot\gamma_\ell\right)^{i}\cdot\tilde\gamma_{x_i}^2\cdot\tilde\gamma_{x_{i+1}}^{-2}\cdot
\left(\tilde\gamma_{y}\cdot\gamma_\ell\right)^{-i-1}].
\]
The above relations for $i=0$ become
\[
1=[\gamma_\ell,\tilde\gamma_{x_0}\cdot\gamma_\ell^{-1}\cdot\tilde\gamma_{y}^{-1}]=
[\gamma_\ell,\tilde\gamma_{x_0}^2\cdot\gamma_\ell^{-1}\cdot\tilde\gamma_{y}^{-1}].
\]
As a consequence $[\gamma_\ell,\tilde\gamma_{x_0}]=[\gamma_\ell,\tilde\gamma_{y}]$ and the group is abelian.
\end{proof}

\begin{prop}\label{prop:triple}
Let $\cC$ be a curve as in Proposition{\rm~\ref{prop:construction}\ref{prop:tres}}, $d>3$. Its fundamental group is 
generated by $\gamma_\ell,\tilde\gamma_{x_0},\tilde\gamma_{x_1},\tilde\gamma_{x_2}$, where $\gamma_\ell$ is central
and $\tilde\gamma_{x_0}\cdot\tilde\gamma_{x_1}\cdot\tilde\gamma_{x_2}\cdot\gamma_\ell^d=1$.
The relations $\tilde\gamma_{x_0}\cdot\tilde\gamma_{x_1}\cdot\tilde\gamma_{x_2}=\tilde\gamma_{x_1}\cdot
\tilde\gamma_{x_2}\cdot\tilde\gamma_{x_0}=\tilde\gamma_{x_2}\cdot\tilde\gamma_{x_0}\cdot\tilde\gamma_{x_1}$ correspond 
to the ordinary triple point of intersection of the three lines.
\end{prop}

\begin{proof}
This fundamental group is the quotient of $\tilde{K}$ obtained by \emph{killing} $\tilde\gamma_{y}$
and $\tilde\gamma_{x_j}$, $j\geq 3$.
The generators are $\gamma_\ell,\tilde\gamma_{x_0},\tilde\gamma_{x_1},\tilde\gamma_{x_2}$ and the relations are:
\begin{enumerate}[label=\rm($\hat{\text{K}}$\arabic*)]
\item $\prod_{j=0}^{2}(\tilde\gamma_{x_j}\cdot\gamma_\ell)\cdot\gamma_\ell^{d-3}=1$,
\item\label{k2} $\left[(\gamma_\ell\cdot\tilde\gamma_{x_i})^d,\tilde\gamma_{x_i}\right]=1$
for $0\leq i<3$,
\item\label{k3} $1=[\gamma_\ell,(\gamma_\ell^{i+1}\cdot\tilde\gamma_{x_i}\cdot\gamma_\ell^{-i})^j\cdot
\gamma_\ell^{-1}\cdot (\gamma_\ell^{i+2}\cdot\tilde\gamma_{x_{i+1}}\cdot\gamma_\ell^{-i-1})^{-j}]$
for $0<j<d$ and $0\leq i<2$,
\item $1=[\gamma_\ell,\tilde\gamma_{x_2}]$,
\item $1=[\gamma_\ell,\tilde\gamma_{x_0}]$.
\end{enumerate}
Let us study the relations \ref{k3}. For $i=0$ and $j=1$ we obtain $[\gamma_\ell,\tilde\gamma_{x_1}]=1$.
\end{proof}

\begin{cor}\label{cor:sigma1}
Let $\cC\in\hat{\Sigma}^{(d)}_1$, $d>3$. Its fundamental group is abelian, that is, $\pi_1(\PP^2\setminus \cC)=\ZZ^3$.
\end{cor}

\begin{proof}
Let us consider the family $\{\cC_t\}_{t\in[0,1]}$ of Proposition~\ref{prop:construction}\ref{prop:tres}. 

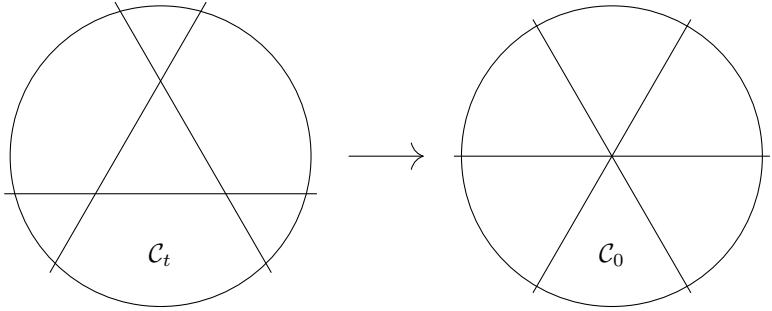
\begin{figure}[ht]
\begin{center}
\begin{tikzpicture}
\draw (0,0) circle [radius=2cm] node[below=1.05cm] {$\mathcal{C}_t$};
\draw ($1.7*(0,1)-.7*({-sqrt(3)/2},-1/2)$)--($-.7*(0,1)+1.7*({-sqrt(3)/2},-1/2)$);
\draw ($1.7*(0,1)-.7*({sqrt(3)/2},-1/2)$)--($-.7*(0,1)+1.7*({sqrt(3)/2},-1/2)$);
\draw ($1.7*({sqrt(3)/2},-1/2)-.7*({-sqrt(3)/2},-1/2)$)--($-.7*({sqrt(3)/2},-1/2)+1.7*({-sqrt(3)/2},-1/2)$);

\draw[-{[scale=2]>}] (2.5,0) -- (3.5,0);

\draw (6,0) circle [radius=2cm] node[below=1.05cm] {$\mathcal{C}_0$};
\draw ($(6,0)+(60:2.1)$)--($(6,0)-(60:2.1)$);
\draw ($(6,0)+(-60:2.1)$)--($(6,0)-(-60:2.1)$);
\draw ($(6,0)+(0:2.1)$)--($(6,0)-(0:2.1)$);
\end{tikzpicture}
\caption{Degeneration around the ordinary triple point}
\label{fig:deg1}
\end{center} 
\end{figure}

For $t$ small enough, these curves are equisingular outside a small ball centered at the ordinary triple point
of $\cC_0$, see Figure~\ref{fig:deg1}. The fundamental group of $\cC_t$ is obtained from the group in Proposition~\ref{prop:triple} 
by replacing the relations $\tilde\gamma_{x_0}\cdot\tilde\gamma_{x_1}\cdot\tilde\gamma_{x_2}=
\tilde\gamma_{x_1}\cdot\tilde\gamma_{x_2}\cdot\tilde\gamma_{x_0}=\tilde\gamma_{x_2}\cdot\tilde\gamma_{x_0}\cdot\tilde\gamma_{x_1}$ 
by $[\tilde\gamma_{x_0},\tilde\gamma_{x_1}]=[\tilde\gamma_{x_1},\tilde\gamma_{x_2}]=[\tilde\gamma_{x_2}\cdot\tilde\gamma_{x_0}]=1$, 
and the result follows.
\end{proof}

\section{An extended family of arithmetic Zariski tuples and an open question}
\label{sec:consecuencias}

It is worth mentioning that recently, the equisingular stratum $\hat\Sigma^{(d)}$ has been generalized as follows
(see~\cite{Sh:18,Sh:19}). Fix $d\geq 3$ and consider three ordered tuples of positive integers
$\mathbf{a}:=(a_1,\dots,a_{n_1})$, $\mathbf{b}:=(b_1,\dots,b_{n_2})$, and $\mathbf{c}:=(c_1,\dots,c_{n_3})$ such that
$d=\sum_i\mathbf{a}=\sum_j\mathbf{b}=\sum_k\mathbf{c}$.

An \emph{Artal-Shirane curve} (defined as Artal curve in~\cite{Sh:18}) of type $(d;\mathbf{a},\mathbf{b},\mathbf{c})$ is a union of a smooth curve of degree~$d$ 
and three lines $L_{\mathbf{a}},L_{\mathbf{b}},L_{\mathbf{c}}$ such that the smooth curve intersects the line 
$L_{\mathbf{a}}$ (resp.~$L_{\mathbf{b}}$, $L_{\mathbf{c}}$) at $n_1$ (resp.~$n_2$, $n_3$) points with intersection 
multiplicities $a_i$, (resp.~$b_j$, $c_k$). 

For $s:=\gcd(\mathbf{a},\mathbf{b},\mathbf{c})>1$ Shirane proved that these curves provide Zariski tuples related 
to $s$-roots of unity. He used an invariant called \emph{splitting graph} based on the linking invariant described 
in \S\ref{sec:link}. 

Moreover, the following generalization of Theorem~\ref{thm:complement} on the homeomorphism type of the complements
of such curves can be stated and proven with the same arguments.

\begin{thm}
Let $\cC_1,\cC_2$ be two Artal-Shirane curves of the same type, associated with distinct and non-conjugate roots of unity. 
Then, $\mathbb{P}^2\setminus\cC_1$ and $\mathbb{P}^2\setminus\cC_2$ are not homeomorphic.
\end{thm}

The interesting fact about these curves lies on the fact that the fundamental groups of their complements is as 
simple as it can be. This can be obtained as a consequence of our Theorem~\ref{thm:sigmazeta}, 
Corollary~\ref{cor:sigma1}, Theorem~\ref{thm:deg2} and an equisingular deformation via the following very 
useful result by Zariski proved in Dimca's book.

\begin{thm}[{\cite[Corollary 3.2]{dimca:li}}]\label{thm:deg2}
Let $\{\cC_t\}_{t\in[0,1]}$ be a continuous family of projective plane curves such that $\cC_0$ is reduced
and the family is equisingular for $(0,1]$. Then, there is an epimorphism 
$\pi_1(\bp^2\setminus\cC_0)\twoheadrightarrow\pi_1(\bp^2\setminus\cC_1)$.
\end{thm}

\begin{cor}
Let $\cC$ be an Artal-Shirane curve of type $(d;\mathbf{a},\mathbf{b},\mathbf{c})$ where either $d>3$ or 
$(d;\mathbf{a},\mathbf{b},\mathbf{c})\neq(3;(3),(3),(3))$ (with non-aligned intersection points).
Then the fundamental group of $\cC$ is abelian.
\end{cor}

\begin{proof}
It is clear that we can deform in an equisingular curve such a curve to a curve of type 
$(d;(d),(d),(d))$ which is in some $\hat{\Sigma}^{(d)}_\zeta$; if $d=3$, then $\zeta\neq 1$.
The result is a direct consequence of Theorem~\ref{thm:sigmazeta}, Corollary~\ref{cor:sigma1}
and Theorem~\ref{thm:deg2}.
\end{proof}

Turning out attention back to the original equisingular family of irreducible curves $\Sigma^{(d)}$
the following question remains open.

\begin{question}
Given two irreducible curves $\cC_i\in \Sigma^{(d)}_{\zeta_i}$, $i=1,2$ where $\zeta_i\in \TT^+_d$, 
$\zeta_1\neq \zeta_2$.
Are the pairs $(\PP^2,\cC_1)$ and $(\PP^2,\cC_2)$ topologically equivalent?
Are their complements homeomorphic, or their fundamental groups isomorphic?
\end{question}

\bibliographystyle{amsplain}

\providecommand{\bysame}{\leavevmode\hbox to3em{\hrulefill}\thinspace}
\providecommand{\MR}{\relax\ifhmode\unskip\space\fi MR }
\providecommand{\MRhref}[2]{%
  \href{http://www.ams.org/mathscinet-getitem?mr=#1}{#2}
}
\providecommand{\href}[2]{#2}

\end{document}